\documentclass[11pt]{article}	% Estilo del documento.
\usepackage[margin=1.2in]{geometry}
\usepackage{authblk}
\usepackage{amscd,amsfonts,amssymb,amsmath,amsthm,latexsym}
\usepackage[pctex32]{graphics}			% Paquete de graficacion  	    	% Silabeo en espaÃ±ol
\usepackage{times}				        % otro tipo de letra (times, bookman)                       			        % Escritura de numeros con punto decimal                     		% Acentos y tildes
\usepackage[]{amssymb}				    % SÃ­mbolos y fuentes
\usepackage[]{amsfonts}				    % de la American Mathematic Society
\usepackage[]{amsmath}			     	% Otras funciones matemÃ¡ticas   			
\providecommand{\bysame}			    % pone subrayado largo en los autores de 		    	
{\makebox[3em]{\hrulefill}\thinspace}	% \bibitem cuando se repiten los autores
\usepackage[Sonny]{fncychap}   		% Encabezados: Lenny, Rejne, Sonny, Glenn, Conny, Bjarne

\newtheorem{definition}{Definiton}[section]
\newtheorem{remark}{Remark}[section]
\newtheorem{proposition}{Proposition}[section]
\newtheorem{lemma}{Lemma}[section]
\newtheorem{theorem}{Theorem}[section]
\newtheorem{corollary}{Corollary}[section]
\newcommand{\RP}{{\mathbb{RP}^1}}

\usepackage[]{fancyhdr}			                            % Encabezados como en el libro de Lamport
\fancypagestyle{plain}{					        % Al iniciar capitulo evita
   \fancyhead{}                         	    % el encabezado superior
   }		    % junto con su cornisa.

\newcommand{\hap}{ H_{\alpha,p} }

\begin{document}

\author{David Damanik\footnote{damanik@rice.edu, D.\ D.\ was supported in part by NSF grant DMS--1700131 and by an Alexander von Humboldt Foundation research award.} , \,Rafael del Rio\footnote{delriomagia@gmail.com, R.\ R.\ was partially supported by project PAPIIT IN 110818} \, and Asaf L. Franco \footnote{asaflevif@gmail.com, A.\ L. \ F. was supported by CONACYT and  PAPIIT IN 110818 }\\
Department of Mathematics, Rice University$^*$\\
Houston, TX 77005, USA\\
\hfill\\
IIMAS - UNAM$^{\dagger\ddagger}$\\
Circuito escolar, Ciudad universitaria 04510 CDMX, M\'exico
}

\title{Random Sturm-Liouville Operators with Generalized Point Interactions}
\date{}

\maketitle

\begin{abstract}
In this work we study the point spectra of selfadjoint Sturm-Liouville operators with generalized point interactions, where the two one-sided limits of the solution data are related via a general $\mathrm{SL}(2,\mathbb{R})$ matrix. We are particularly interested in the stability of eigenvalues with respect to the variation of the parameters of the interaction matrix. As a particular application to the case of random generalized point interactions we establish a version of Pastur's theorem, stating that except for degenerate cases, any given energy is an eigenvalue only with probability zero. For this result, independence is important but identical distribution is not required, and hence our result extends Pastur's theorem from the ergodic setting to the non-ergodic setting.
\end{abstract}

Mathematics Subject Classification (MSC2010): 34L05, 47E05, 47N99.

\section{Introduction}\label{s.1}

In this paper we study the point spectra of selfadjoint Sturm-Liouville operators with generalized point interactions. More specifically, we investigate whether varying the parameters of the spectral problem preserves or destroys the fact that a given energy is an eigenvalue. This is of particular interest in the setting of random parameters. In the case of i.i.d.\ random variables, one can use methods from ergodic theory and it is a classical result due to Pastur \cite{P80} that a given energy can be an eigenvalue only with probability zero. However, if the random variables are not identically distributed, Pastur's argument does not apply and it was realized only recently, in the special case of $\delta$ and $\delta'$ point interactions, that a result in the same spirit still holds \cite{RRAL}.

\medskip

The purpose of the present paper is two-fold. On the one hand, we introduce a new approach to this problem, which is based on geometric ideas and mapping properties of $\mathrm{SL}(2,\mathbb{R})$ matrices. This makes the resulting spectral statement particularly natural and easy to understand. On the other hand, our approach allows us to generalize the setting and pass from $\delta$ and $\delta'$ point interactions to the whole class of real connecting selfadjoint point interactions and hence develops the theory in the appropriate level of generality.

\medskip

The key idea will be the following. Fixing the boundary conditions of the spectral problem and considering an energy that is an eigenvalue for a given collection of parameters, we vary one of them while keeping the others fixed. How to vary the parameter is clear if $\delta$ or $\delta'$ point interactions are considered, but it is somewhat less clear in the case of general $\mathrm{SL}(2,\mathbb{R})$ matrices connecting the left- and right-limit of the solution data at the point in question. To this end, we will consider the Iwasawa decomposition of an $\mathrm{SL}(2,\mathbb{R})$ matrix, which expresses it as a canonical product of a parabolic, a hyperbolic, and an elliptic factor. This provides the parameters we seek and will vary. The next step is to investigate the stability question for the eigenvalue problem at hand when the parameter is varied. It turns out in most cases that there is a dichotomy. Either the eigenvalue is present for all values of the parameter, or it is present only for the one we started with and not for any other value. To establish this dichotomy we look at the projective action of the $\mathrm{SL}(2,\mathbb{R})$ matrix in question and are able to exhibit this dichotomy via direct and very simple calculations. Once the dichotomy corresponding to a single point interaction has been established, it will then be straightforward to process the entire family and to deduce a global result. The application to the case of random parameters is then also immediate.

\medskip

Since they are crucial to our discussion, we will include discussions of the essential tools we use in Section~\ref{s.2}, even though this material is well known. We hope that this will be useful for those readers who are less familiar with these tools in the context of spectral theory applications. This includes in particular the Iwasawa decomposition of $\mathrm{SL}(2,\mathbb{R})$ matrices and their mapping properties on the real projective line. As a warm-up we consider the case of a single $\delta$ interaction in Section~\ref{s.3}. Although this case has been studied before, we present our new perspective in this simple setting, partly to introduce the ideas, and partly to show how the known result can be proved with our method. In Section~\ref{s.4} we then consider the case of a general connecting point interaction, which is given by an $\mathrm{SL}(2,\mathbb{R})$ matrix. The three parameters describing such a matrix are given, in our representation, by the parameters corresponding to the three factors in the Iwasawa decomposition of the given matrix. We discuss the stability question for a given eigenvalue when two of the three parameters are fixed and the third is varied. Next, Section~\ref{s.5} considers the case of countably many general point interactions located on a discrete set inside the interval. Again, only one parameter for one interaction will be varied, while all other parameters are fixed, and the eigenvalue stability problem is investigated. Finally, we consider the case of countably many general point interactions with \emph{random} parameters in Section~\ref{s.6} and prove a result in the spirit of Pastur and in the appropriate level of generality, that is, without assuming identical distribution. We do, however, make crucial use of independence.

\section*{Acknowledgments} Part of this work was done during a visit of D.D.\ in March 2019 to UNAM, which was partially supported by PAPIIT-UNAM.

\section{Preliminaries}\label{s.2}

In this section we collect a few tools, all of which are well known. As usual $\mathrm{SL}(2,\mathbb{R})$ and $\mathrm{GL}(2,\mathbb{R})$ denote the special and general linear groups respectively.
We include this material for the sake of the reader. Anyone familiar with these concepts may skip ahead to the next section.

\subsection{Transfer Matrices}

Let us discuss an elementary way to introduce the transfer matrices, which we emphasize is not the standard way of introducing them.

Consider an open interval $I = (a,b) \subset \mathbb{R}$, an $L^1_\mathrm{loc}$ potential $V : I \to \mathbb{R}$, and an energy $E \in \mathbb{R}$. The associated differential equation is
\begin{equation}\label{e.diffequ}
-u''(x) + V(x) u(x) = E u(x), \quad x \in I.
\end{equation}
Standard ODE theory shows that for each $x \in I$ and each $(v,d)^T \in \mathbb{R}^2$, there is a unique solution $u$ of \eqref{e.diffequ} with $(u(x),u'(x))^T = (v,d)^T$. Moreover, all real solutions of \eqref{e.diffequ} arise in this way. See for example \cite[Thm. 2.2.1]{AZ}. This has the following immediate consequence.

\begin{proposition}\label{l.solprojx}
The set $S_E$ of real solutions of \eqref{e.diffequ} is a two-dimensional real vector space and, for each $x \in I$, the map
$$
M_{x,E}: S_E \to \mathbb{R}^2, \quad u \mapsto \begin{pmatrix} u(x) \\ u'(x) \end{pmatrix}
$$
is a linear isomorphism.
\end{proposition}

\begin{proof}
It follows directly from the definition of the map $M_{x,E}$ (and the linearity of differentiation) that it is linear. By the standard ODE results quoted above, it is both onto and one-to-one. This also implies the well-known fact that $S_E$ is a two-dimensional real vector space.
\end{proof}

\begin{proposition}\label{MSL}
For $x,y \in I$, there is a matrix $M(x,y;E) \in \mathrm{SL}(2,\mathbb{R})$ such that for every $u \in S_E$, we have
\begin{equation}\label{e.transfermatrix}
\begin{pmatrix} u(x) \\ u'(x) \end{pmatrix} = M(x,y;E) \begin{pmatrix} u(y) \\ u'(y) \end{pmatrix}.
\end{equation}
\end{proposition}

\begin{proof}
If we define $M(x,y;E) := M_{x,E} M_{y,E}^{-1}$, then \eqref{e.transfermatrix} holds by Proposition~\ref{l.solprojx}. By construction, $M(x,y;E) \in \mathrm{GL}(2,\mathbb{R})$, so it remains to show that $\det M(x,y;E) = 1$.

Consider the two solutions $u_D, u_N \in S_E$ with
$$
\begin{pmatrix} u_N(y) & u_D(y) \\ u'_N(y) & u'_D(y) \end{pmatrix} = \begin{pmatrix} 1 & 0 \\ 0 & 1 \end{pmatrix}.
$$
Then,
\begin{align*}
M(x,y;E) & = M(x,y;E) \begin{pmatrix} 1 & 0 \\ 0 & 1 \end{pmatrix} \\
& = M(x,y;E) \begin{pmatrix} u_N(y) & u_D(y) \\ u'_N(y) & u'_D(y) \end{pmatrix} \\
& = \begin{pmatrix} u_N(x) & u_D(x) \\ u'_N(x) & u'_D(x) \end{pmatrix},
\end{align*}
and therefore
\begin{align*}
\det M(x,y;E) & = \det \begin{pmatrix} u_N(x) & u_D(x) \\ u'_N(x) & u'_D(x) \end{pmatrix} \\
& = u_N(x) u'_D(x) - u_D(x) u'_N(x) \\
& = u_N(y) u'_D(y) - u_D(y) u'_N(y) \\
& = 1.
\end{align*}
Here we used the constancy of the Wronskian, which follows from the fact that $u_D, u_N$ solve \eqref{e.diffequ}:
\begin{align*}
(u_N(t) u'_D(t) & - u_D(t) u'_N(t))' = \\
& = u_N'(t) u'_D(t) + u_N(t) u''_D(t) - u_D'(t) u'_N(t) - u_D(t) u''_N(t)\\
& = u_N(t) [(V(t) - E)u_D(t)] - u_D(t) [(V(t) - E)u_N(t)] \\
& = 0.
\end{align*}
\end{proof}

\subsection{The Real Projective Line}

Recall that the real projective line $\RP$ is given by
$$
\RP = \{ \text{lines in $\mathbb R^2$ through the origin} \}.
$$
Note that the elements of $\RP$ are equivalence classes with respect to the equivalence relation on $\mathbb R^2 \setminus \{ 0 \}$ given by
$$
v \sim w \quad \Leftrightarrow \quad \exists \lambda \in \mathbb R \setminus \{0\} : v = \lambda w.
$$

\begin{definition}\label{eqqclass}
We denote the equivalence class of $v \in \mathbb R^2 \setminus \{ 0 \}$ by $[v]$.
\end{definition}

\begin{remark}\label{linearg}
Let $u=(u_1,u_2)^T$ and $v=(v_1,v_2)^T$. Then,
$[u]=[v]$ if and only if $arg(u_2+iu_1)=arg(v_2+iv_1 )+k\pi$, $k\in\mathbb{Z}$.
\end{remark}

\begin{lemma}\label{DMM}
Any $M \in \mathrm{GL}(2,\mathbb R)$ induces a well-defined bijective map from $\RP$ to $\RP$, which will be denoted by $\tilde M$, via
$$
\tilde M ([v]) = [Mv].
$$
\end{lemma}

\begin{proof}
Let $u\sim v$. Then $u = \lambda v$ for some $\lambda \in \mathbb R \setminus \{ 0 \}$ and
$$
[M u] = \tilde M [u] = \tilde M [\lambda v] = [M \lambda v] = [\lambda M v] = [Mv].
$$
This shows that $\tilde M$ is well defined.

Let $[v] \in \RP$ with representative $v$. Since $M$ is surjective by assumption, there exists $u \in \mathbb R^2$ such that $Mu = v$. Since
$$
\tilde M ([u]) = [Mu] = [v],
$$
it follows that $\tilde M$ is surjective.\\

Finally, suppose $[Mu] = [Mv]$. Then there exists $k \in \mathbb R \setminus \{ 0 \}$ such that $Mu = k Mv$, and since $M$ is injective by assumption, $u = kv$. Thus $[u]=[v]$ and $\tilde M$ is injective.
\end{proof}

\subsection{The Iwasawa Decomposition of $\mathrm{SL}(2,\mathbb R)$ Matrices}

In this subsection we discuss the Iwasawa decomposition of $\mathrm{SL}(2,\mathbb R)$ matrices; compare \cite{L85}. We provide some details on how to obtain this decomposition for the reader's convenience.

\medskip

We define the following subgroups of $\mathrm{SL}(2,\mathbb R)$:
\begin{align*}
\mathcal{E} & = \left \{ E_\theta := \begin{pmatrix} \cos \theta & -\sin \theta \\ \sin \theta & \cos \theta \end{pmatrix} : \theta \in \mathbb{R} \right\}, \\
\mathcal{P} & = \left \{ P_\alpha := \begin{pmatrix} 1 & \alpha \\ 0 & 1 \end{pmatrix} : \alpha \in \mathbb R \right\}, \\
\mathcal{H} & = \left \{ H_r := \begin{pmatrix} r & 0 \\ 0 & 1/r \end{pmatrix} : r > 0 \right\}.
\end{align*}

\begin{theorem}[Iwasawa Decomposition]
Every $A\in \mathrm{SL}(2,\mathbb R)$ can be written in a unique way as $A = P_\alpha H_r E_\theta$, where $P_\alpha \in \mathcal{P}$, $H_r\in\mathcal{H}$ and $E_\theta \in\mathcal{E}$.
\end{theorem}

\begin{proof}
Consider the complex upper half-plane, $\mathbb{C}_+ = \{ z \in \mathbb C : \Im z > 0 \}$. Given $A \in \mathrm{SL}(2,\mathbb R)$, we consider its action on $\mathbb C_+$ given by
$$
A \cdot z = \begin{pmatrix} a & b \\ c & d \end{pmatrix} \cdot z := \frac{az + b}{cz + d}.
$$
Note that $A \cdot z$ indeed belongs to $\mathbb C_+$ for each $z \in \mathbb C_+$ since
$$
\Im \left( \frac{az + b}{cz + d} \right) = \frac{(ad-bc) \Im z}{|cz + d|^2} = \frac{\Im z}{|cz + d|^2} > 0.
$$
Moreover, note that
\begin{equation}\label{e.multi}
(A \cdot B) \cdot z = A \cdot (B \cdot z)
\end{equation}
for all $A, B  \in \mathrm{SL}(2,\mathbb R)$ and $z \in \mathbb C_+$.

Consider the case $A \cdot i = i$, that is,
$$
\frac{ai + b}{ci + d} = i \Leftrightarrow ai + b = di - c \Leftrightarrow a = d \text{ and } b = -c.
$$
Thus the condition $\det A = ad - bc = 1$ becomes $a^2 + c^2 = 1$ and we can choose $\theta \in \mathbb R$ with $a = \cos \theta$ and $c = \sin \theta$, so that
$$
A = \begin{pmatrix} a & b \\ c & d \end{pmatrix} = \begin{pmatrix} a & -c \\ c & a \end{pmatrix} = \begin{pmatrix} \cos \theta & -\sin \theta \\ \sin \theta & \cos \theta \end{pmatrix}.
$$
This discussion shows that $A \cdot i = i$ if and only if $A \in \mathcal{E}$.

Given any $A \in \mathrm{SL}(2,\mathbb R)$, we consider $A \cdot i \in \mathbb C_+$ and set
$$
\alpha := \Re (A \cdot i), \quad r := (\Im (A \cdot i))^{1/2}.
$$
Then,
\begin{align*}
A \cdot i & = \alpha + i r^2 \\
& = \begin{pmatrix} r & \alpha/r \\ 0 & 1/r \end{pmatrix} \cdot i \\
& = \begin{pmatrix} 1 & \alpha \\ 0 & 1 \end{pmatrix}  \begin{pmatrix} r & 0 \\ 0 & 1/r \end{pmatrix}  \cdot i
\end{align*}
Thus, by \eqref{e.multi},
$$
\begin{pmatrix} r & 0 \\ 0 & 1/r \end{pmatrix}^{-1} \begin{pmatrix} 1 & \alpha \\ 0 & 1 \end{pmatrix}^{-1} A \cdot i = i,
$$
which implies that
$$
\begin{pmatrix} r & 0 \\ 0 & 1/r \end{pmatrix}^{-1} \begin{pmatrix} 1 & \alpha \\ 0 & 1 \end{pmatrix}^{-1} A = \begin{pmatrix} \cos \theta & -\sin \theta \\ \sin \theta & \cos \theta \end{pmatrix}
$$
for a suitable $\theta \in \mathbb R$ by our discussion above. Thus,
$$
A = \begin{pmatrix} 1 & \alpha \\ 0 & 1 \end{pmatrix} \begin{pmatrix} r & 0 \\ 0 & 1/r \end{pmatrix} \begin{pmatrix} \cos \theta & -\sin \theta \\ \sin \theta & \cos \theta \end{pmatrix},
$$
as desired. This establishes existence.

To show uniqueness, consider the identity
$$
\begin{pmatrix} 1 & \alpha_1 \\ 0 & 1 \end{pmatrix} \begin{pmatrix} r_1 & 0 \\ 0 & 1/r_1 \end{pmatrix} \begin{pmatrix} \cos \theta_1 & -\sin \theta_1 \\ \sin \theta_1 & \cos \theta_1 \end{pmatrix} = \begin{pmatrix} 1 & \alpha_2 \\ 0 & 1 \end{pmatrix} \begin{pmatrix} r_2 & 0 \\ 0 & 1/r_2 \end{pmatrix} \begin{pmatrix} \cos \theta_2 & -\sin \theta_2 \\ \sin \theta_2 & \cos \theta_2 \end{pmatrix}
$$
with $\alpha_1, \alpha_2, \theta_1, \theta_2 \in \mathbb R$ and $r_1, r_2 > 0$.

Applying both sides to $i \in \mathbb C_+$, we obtain
$$
\begin{pmatrix} 1 & \alpha_1 \\ 0 & 1 \end{pmatrix} \begin{pmatrix} r_1 & 0 \\ 0 & 1/r_1 \end{pmatrix} \cdot i = \begin{pmatrix} 1 & \alpha_2 \\ 0 & 1 \end{pmatrix} \begin{pmatrix} r_2 & 0 \\ 0 & 1/r_2 \end{pmatrix} \cdot i,
$$
which (by an observation above) is equivalent to
$$
\alpha_1 + i r_1^2 = \alpha_2 + i r_2^2.
$$
This implies $\alpha_1 = \alpha_2$ and $r_1 = r_2$ (since $r_1, r_2 > 0$). Once this holds, we must also have
$$
\begin{pmatrix} \cos \theta_1 & -\sin \theta_1 \\ \sin \theta_1 & \cos \theta_1 \end{pmatrix} = \begin{pmatrix} \cos \theta_2 & -\sin \theta_2 \\ \sin \theta_2 & \cos \theta_2 \end{pmatrix},
$$
proving uniqueness.
\end{proof}

\begin{remark}
Since any matrix in $\mathrm{SL}(2,\mathbb R)$ can be written as the inverse of the transpose of a matrix in $\mathrm{SL}(2,\mathbb R)$, we also have the decomposition
$$
A = \begin{pmatrix} 1 & 0 \\ -\tilde\alpha & 1 \end{pmatrix} \begin{pmatrix} \frac{1}{\tilde r} & 0 \\ 0 & \tilde r \end{pmatrix} \begin{pmatrix} \cos \tilde\theta & -\sin \tilde\theta \\ \sin \tilde\theta & \cos \tilde\theta \end{pmatrix}
$$
for some $\tilde\alpha\in\mathbb{R}$, $\tilde r>0$ and $\tilde \theta \in \mathbb{R}$.
\end{remark}

\subsection{The Differential Operator and its Eigenvalues}

For a finite closed interval $I=[a,b]$ and a real-valued $V \in L^1(I)$, consider the associated differential expression defined by
$$
\tau f := -f'' + Vf.
$$

For all $x,y \in I$, let $M(x,y;E)$ be the transfer matrix defined in Proposition~\ref{MSL}. Then $M(x,y;E) \in \mathrm{SL}(2,\mathbb{R})$ and for every real solution of $\tau u = Eu$, we have
\[
\left(\begin{array}{c}
      u(x) \\
     u'(x)
\end{array}\right)=M(x,y;E)\left(\begin{array}{c}
     u(y) \\
     u'(y)
\end{array}\right).
\]

Let $T_{\theta,\gamma}$ be the selfadjoint operator defined by
$$
T_{\theta,\gamma} f = \tau f
$$
with domain
$$
D(T_{\theta,\gamma}) := \{ f \in L^2(I): f, f' \mbox{abs. con. on} \, I, \tau f \in L^2(I)
$$
$$
f(a) \cos \theta - f'(a) \sin \theta = 0
$$
$$
f(b) \cos \gamma - f'(b) \sin \gamma = 0 \}.
$$

As an application of Lemma~\ref{DMM} we will prove the following well-known result.

\begin{theorem}
Let $E \in \mathbb R$, then for each $\theta \in [0,\pi)$ $(\gamma \in [0,\pi))$, there exists a unique $\gamma \in [0,\pi)$ $(\theta\in [0,\pi))$ such that $E \in \sigma_p(T_{\theta,\gamma})$.
\end{theorem}

\begin{proof}
For $E \in \mathbb{R}$ and $\theta \in [0,\pi)$, there exists a non-trivial solution $u \in L^2(I)$ of $\tau u = E u$, which is unique up to a non-zero multiple, satisfying
$$
u(a) \cos \theta - u'(a) \sin \theta = 0.
$$
Since $M(b,a;E) \in \mathrm{SL}(2,\mathbb{R})$, there exists a unique vector $(u(b),u'(b))^T$ satisfying
\[
\left(\begin{array}{c}
      u(b) \\
     u'(b)
\end{array}\right) = M(b,a;E)\left(\begin{array}{c}
     u(a) \\
     u'(a)
\end{array}\right)
\]
Let $\gamma := \arctan \frac{u(b)}{u'(b)}$. Then,
$$
u(b) \cos \gamma - u'(b) \sin \gamma = 0.
$$
Therefore, $E \in \sigma_p(T_{\theta,\gamma})$.

Assume $\tilde \gamma \in [0,\pi)$, $\tilde \gamma \not= \gamma$ and $E \in \sigma_p(T_{\theta,\tilde\gamma})$. Then there exists a non-zero $v \in D(T_{\theta,\tilde\gamma})$ such that $\tau v = E v$,
$$
v(a) \cos \theta - v'(a) \sin \theta = 0,
$$
$$
v(b) \cos \tilde \gamma - v'(b) \sin \tilde \gamma = 0.
$$
Thus the angle of the vector $(v(a),v'(a))^T$ is $\theta$ and the angle of the vector $(v(b),v'(b))^T$ is $\tilde \gamma$. Then by Lemma \ref{DMM},
$$
\tilde M(a,b;E) [(v(b),v'(b))^T] = [(u(a),u'(a))^T] = \tilde M(a,b;E) [(u(b),u'(b))^T]
$$
the vectors $(v(b),v'(b))^T$ and $(u(b),u'(b))^T$ must belong to the same element of the real projective line, i.e.\ they must have the same angle, so that $\gamma = \tilde \gamma$.
Analogously, for each $\gamma \in [0,\pi)$, there exists a unique $\theta \in [0,\pi)$ such that $E \in \sigma_p (T_{\theta,\gamma})$.
\end{proof}

\begin{corollary}
If $E\in \sigma_p(T_{\theta,\gamma})$, then $E \not\in \sigma_p (T_{\tilde\theta,\gamma})$ for every $\tilde \theta \in [0,\pi) \setminus \{ \theta \}$.
\end{corollary}

\section{The Case of a Single $\delta$-Interaction}\label{s.3}

As a warm-up we consider the case of a single $\delta$-interaction.

Let $I = [a,b] \subset \mathbb R$ be a closed finite interval, $V \in L^1(I)$ real valued, $p \in J$ an interior point, and $\alpha \in \mathbb{R}$.

We consider the formal differential expressions
$$
\tau := -\frac{d^2}{dx^2} + V
$$
and
$$
\tau_{\alpha,p} := -\frac{d^2}{dx^2} + V + \alpha \delta (x-p).
$$

The maximal operator $T_{\alpha,p}$ corresponding to $\tau_{\alpha,p}$ is defined by
$$
T_{\alpha,p} f = \tau f
$$
$$
D(T_{\alpha,p}) = \Big\{ f \in L^2(I) : \,f,\,f'\mbox{ abs. cont in } J \backslash \{p\}, -f'' + V f \in L^2(J),
$$
\[
\left(\begin{array}{c}
     f(p+) \\
     f'(p+)
\end{array}\right)=A_{\alpha,p}\left(\begin{array}{c}
f(p-)\\
f'(p-)\end{array}\right) \Big\}.
	\]
	
Here, $A_{\alpha,p}$ is the $\mathrm{SL}(2,\mathbb{R})$ matrix defined by
\begin{equation}\label{e.deltainteraction}
A_{\alpha,p} = \left(\begin{array}{cc} 1 & 0 \\ \alpha & 1 \end{array}\right).
\end{equation}

%Suppose now that $\tau_{\alpha,p}$ is regular at $a$ and $b$, i.e $a$ and $b$ are finite and $V \in L^1([a,b])$.
Let us consider the selfadjoint restriction $H_{\alpha,p}$ of $T_{\alpha,p}$ in $L^2(I)$, see Theorem 5.2 in \cite{BSW}, defined by
\begin{equation}\label{halfa}
H_{\alpha,p} f = \tau f
\end{equation}
\[
\begin{array}{ccc} D(H_{\alpha,p}) & = & \left\{f\in D(T_{\alpha,p}) : \begin{array}{c} f(a) \cos \theta + f'(a) \sin \theta = 0 \\	 {f(b) \cos \gamma + f'(b) \sin \gamma = 0} \end{array} \right\} \qquad \qquad  \theta,\,\gamma\in [0,\pi). \end{array}
\]
	
\begin{theorem}
Let $E \in \sigma_p(H_{\alpha,p})$. Then one of the following holds:
\begin{itemize}
    \item [$i)$] $E \in \sigma_p (H_{\tilde \alpha, p})$ for every $\tilde \alpha \in \mathbb{R}$,
    \item[$ii)$] $E \not\in \sigma_p (H_{\tilde \alpha, p})$ for every $\tilde \alpha \in \mathbb{R} \setminus \{ \alpha \}$.
\end{itemize}
\end{theorem}

\begin{proof}
Note first that on the level of transfer matrices, the local point interaction inserts the factor \eqref{e.deltainteraction} between $M(y,p+;E)$ and $M(p-,x;E)$ for $a \le x < p < y \le b$.

Let $E \in \mathbb{R}$ be such that $E \in \sigma_p(\hap)$. Then there exists a non-zero $u\in D(\hap)$ with $\hap u = E u$. In particular, we have
$$
\begin{pmatrix} u(p+) \\ {u'(p+)}\end{pmatrix}= A_{\alpha,p}\begin{pmatrix} {u(p-)} \\ {u'(p-)} \end{pmatrix}.
$$

Suppose $ii)$ fails; and hence we have to prove $i)$. Let $E \in \sigma_p(H_{\tilde\alpha,p})$ for some $\tilde \alpha \in \mathbb{R} \setminus \{ \alpha \}$. There exists a non-zero $v \in D(H_{\tilde\alpha,p})$ such that $H_{\tilde\alpha,p} v = E v$.
Since $M = M(p-,a;E) \in \mathrm{SL}(2,\mathbb{R})$ and $[(u(a),u'(a))^T] = [(v(a),v'(a))^T]$, we have
$$
[(u(p-),u'(p-))^T] = \tilde M([(u(a),u'(a))^T]) = \tilde M([(v(a),v'(a))^T]) = [(v(p-),v'(p-))^T].
$$
Thus there exists $k \in \mathbb{R} \setminus \{0\}$ such that
$$
\begin{pmatrix} u(p-) \\ u'(p-) \end{pmatrix} = k \begin{pmatrix} v(p-) \\ v'(p-) \end{pmatrix},
$$
and since $u \in D(H_{\alpha,p})$ and $v \in D(H_{\tilde\alpha,p})$,
$$
\begin{pmatrix} u(p-) \\ (\alpha-\tilde\alpha) u(p-) + u'(p-) \end{pmatrix} = \begin{pmatrix} 1 & 0 \\ \alpha - \tilde \alpha & 1 \end{pmatrix} \begin{pmatrix} u(p-) \\ u'(p-) \end{pmatrix} = A_{\tilde\alpha,p}^{-1} A_{\alpha,p} \begin{pmatrix} u(p-) \\ u'(p-) \end{pmatrix} =
$$
$$ = A_{\tilde\alpha,p}^{-1} \begin{pmatrix} u(p+) \\ u'(p+) \end{pmatrix} = k A_{\tilde\alpha,p}^{-1} \begin{pmatrix} v(p+) \\ v'(p+) \end{pmatrix} = k \begin{pmatrix} v(p-) \\ v'(p-) \end{pmatrix} = \begin{pmatrix} u(p-) \\ u'(p-) \end{pmatrix}$$
and then $(\alpha-\tilde\alpha)u(p-)+u'(p-)=u'(p-)$. Since $\tilde\alpha\not=\alpha$, $u(p-)=0$. Thus $\forall\tilde\alpha\in\mathbb{R}$,
$$\begin{pmatrix} u(p+) \\ {u'(p+)}\end{pmatrix}=\begin{pmatrix}0\\u'(p+)\end{pmatrix}
=\begin{pmatrix}1&0&\\ \alpha&1\end{pmatrix}\begin{pmatrix} 0 \\ {u'(p-)}\end{pmatrix}=\begin{pmatrix}1&0&\\\tilde\alpha&1\end{pmatrix}\begin{pmatrix} 0 \\ {u'(p-)}\end{pmatrix}=
A_{\tilde\alpha,p}\begin{pmatrix} {u(p-)} \\ {u'(p-)} \end{pmatrix}$$

Therefore $u \in \sigma_p(H_{\tilde\alpha,p})$, $\forall\tilde\alpha\not=\alpha$ and $i)$ holds.
\end{proof}

%%%%%%%%%%%%%%%%%%%%

\section{The Case of a Single General Point Interaction}\label{s.4}

Now we construct the operator with one general point interaction. Let $I = [a,b] \subset \mathbb R$ be a closed finite interval. Let $V \in L^1(I)$ be a real-valued function, $p \in I$ an interior point and $A_{\alpha, r, \theta} \in \mathrm{SL}(2,\mathbb{R})$ with Iwasawa decomposition $A_{\alpha,r,\theta} = P_\alpha H_r E_\theta$, where $P_\alpha \in \mathcal{P}$, $H_r \in \mathcal{H}$ and $E_\theta \in \mathcal{E}$. We consider the formal differential expression
$$
\tau := -\frac{d^2}{dx^2} + V.
$$

The corresponding maximal operator $T_{\alpha,r,\theta}$ is defined by
$$
T_{\alpha,r,\theta} f=\tau f
$$
$$
D(T_{\alpha,r,\theta}) = \Big\{ f\in L^2(I):\,f,\,f'\mbox{ abs. cont in }J\backslash \{p\},-f''+Vf\in L^2(J),
$$
\[
\left(\begin{array}{c}
     f(p+) \\
     f'(p+)
\end{array}\right)=A_{\alpha,r,\theta}\left(\begin{array}{c}
f(p-)\\
f'(p-)\end{array}\right) \Big\}.
	\]
	
%Suppose now that $\tau_{\alpha,r,\theta}$ is regular at $a$ and $b$, that is, $a$ and $b$ are finite and $V\in L^1([a,b])$.
Let us consider the selfadjoint restriction $H_{\alpha,r,\theta}$ of $T_{\alpha,r,\theta}$ in $L^2(I)$, see equation (4.3) in \cite{WZ}, defined by
\begin{equation}\label{halfa2}
H_{\alpha,r,\theta} f=\tau f
\end{equation}
\[
\begin{array}{ccc} D(H_{\alpha, r, \theta}) & = & \left\{ f \in D(T_{\alpha,r,\theta}) : \begin{array}{c} f(a) \cos \delta + f'(a) \sin \delta = 0 \\ {f(b) \cos \gamma + f'(b) \sin\gamma = 0} \end{array} \right\}, \quad \delta,\,\gamma \in [0,\pi). \end{array}
\]

\begin{lemma}\label{lem1}
Let $\theta, \tilde\theta \in \mathbb{R}$ and fix $v \in \mathbb{R}^2$. The following holds: $\tilde\theta \not= \theta + k\pi$, $k \in \mathbb{Z}$ if and only if $[A_{\alpha,r,\theta} v] \not= [A_{\alpha,r,\tilde\theta}v]$.
\end{lemma}

\begin{proof}
\hfill
\begin{itemize}

\item [$\Rightarrow$)] Let $\tilde\theta \not= \theta + k\pi$, $k \in \mathbb{Z}$ and $v\in \mathbb{R}^2$. Since $E_\gamma$ acts as a rotation of $\gamma$ degrees on $v$, we have $[E_\theta v] \not= [E_{\tilde\theta}v]$. Taking into account that $P_\alpha H_r \in \mathrm{SL}(2,\mathbb{R})$, Lemma~\ref{DMM} gives that $[A_{\alpha, r, \theta} v] \not= [A_{\alpha, r, \tilde\theta}v]$.

\item[$\Leftarrow$)] Suppose now  $[A_{\alpha,r,\theta} v] \not= [A_{\alpha,r,\tilde\theta}v]$. Recalling the definition introduced in Lemma \ref{DMM},
$$
\widetilde{P_\alpha H_r} [E_{\tilde\theta} v] = [P_\alpha H_r E_{\tilde\theta} v] = [A_{\alpha, r, \tilde \theta} v] \not= [A_{\alpha,r,\theta} v] = \widetilde{P_\alpha H_r} [E_{\theta} v],
$$
since $\widetilde{P_\alpha H_r}$ is injective. Thus $[E_\theta v] \not= [E_{\tilde\theta}v]$ and hence $\tilde\theta \not= \theta + k\pi$, $k \in \mathbb{Z}$.

\end{itemize}
\end{proof}

\begin{lemma}\label{lem2}
Let $r, \tilde r > 0$, $\tilde r \not= r$ and $v \in \mathbb{R}^2$. The following are equivalent:
\begin{itemize}
    \item [$i)$] $[v] = [(\sin \theta, \cos \theta)^T]$ or $[v] = [(\cos \theta, -\sin \theta)^T]$,
    \item[$ii)$] $[A_{\alpha,r,\theta} v] = [A_{\alpha,\tilde r,\theta}v]$.
\end{itemize}
\end{lemma}

\begin{proof}
\hfill
\begin{itemize}
    \item[$ii)\Rightarrow i)$] Assume $[A_{\alpha,r,\theta} v] = [A_{\alpha,\tilde r,\theta}v]$. Then $\tilde P_\alpha [H_r E_\theta v] = \tilde P_\alpha [H_{\tilde r} E_\theta v]$. Since by Lemma~\ref{DMM} $\tilde P_\alpha$ is injective, $[H_r E_\theta v] = [H_{\tilde r} E_\theta v]$, that is, there exists $k \in \mathbb{R} \setminus \{ 0 \}$ such that $H_r E_\theta v = k H_{\tilde r} E_\theta v$. Thus, $H_{\tilde r}^{-1} H_r E_\theta v = k E_\theta v$, and therefore $E_\theta v$ is eigenvector of the diagonal matrix $H_{\tilde r}^{-1} H_r$. Since $r \not= \tilde r$, the eigenvectors are multiples of $[0,1]^T$ or $[1,0]^T$. Then $[E_\theta v] = [(1,0)^T]$ or $[E_\theta v] = [(0,1)^T]$, taking into account that
    $$
    E_\theta^{-1} = \begin{pmatrix} \cos \theta & \sin \theta \\ -\sin \theta & \cos \theta \end{pmatrix}
    $$
    and Lemma~\ref{DMM}, we obtain $[v] = [(\sin \theta, \cos \theta)^T]$ or $[v] = [(\cos \theta, -\sin \theta)^T]$.

    \item[$i) \Rightarrow ii)$] Assume $[v] = [(\sin \theta, \cos \theta)^T]$. Then,
    $$
    [A_{\alpha,r,\theta} v] = [P_\alpha H_r (0,1)^T] = [\frac{1}{r} P_\alpha (0,1)^T] = [\frac{1}{\tilde r} P_\alpha (0,1)^T] = [P_\alpha H_{\tilde r} (0,1)^T] = [A_{\alpha,\tilde r,\theta} v].
    $$
    When $[v] = [(\cos \theta, -\sin \theta)^T]$, the result follows in an analogous way.

\end{itemize}
\end{proof}

\begin{lemma}\label{lem3}
Let $\alpha, \tilde \alpha \in \mathbb{R}$, $\tilde \alpha \not= \alpha$ and $v \in \mathbb{R}^2$. The following are equivalent:
\begin{itemize}
    \item [$i)$] $[v] = [(\cos \theta, -\sin \theta)^T]$,
    \item[$ii)$] $[A_{\alpha,r,\theta} v] = [A_{\tilde \alpha, r,\theta}v]$.
\end{itemize}
\end{lemma}

\begin{proof}
\hfill
\begin{itemize}
    \item[$ii)\Rightarrow i)$] Assume $[A_{\alpha,r,\theta} v] = [A_{\tilde \alpha, r, \theta} v]$, that is, there exists $k \in \mathbb{R} \setminus \{ 0 \}$ such that $P_\alpha H_r E_\theta v = k P_{\tilde\alpha} H_r E_\theta v$. Then $P_{\tilde\alpha}^{-1} P_\alpha H_r E_\theta v = k H_r E_\theta v$. Thus $H_r E_\theta v$ is an eigenvector of the matrix $P_{\tilde\alpha}^{-1} P_\alpha$. Since $\alpha \not= \tilde\alpha$, $P_{\tilde\alpha}^{-1} P_\alpha \not= I$ and its eigenvectors are multiples of $(1,0)^T$. Then $[H_r E_\theta v] = [(1,0)^T]$, and taking into account that
    $$
    (H_r E_\theta)^{-1} = \begin{pmatrix} \frac{1}{r} \cos \theta & r \sin \theta \\ -\frac{1}{r} \sin \theta & r \cos \theta \end{pmatrix}
    $$
    as well as Lemma~\ref{DMM}, we obtain $[v] = [\frac{1}{r} (\cos \theta, -\sin \theta)^T] = [(\cos \theta, -\sin \theta)^T]$.

    \item[$i) \Rightarrow ii)$] Assume $[v] = [(\cos \theta, -\sin \theta)]$. Then
    $$
    [A_{\alpha, r, \theta} v] = [P_\alpha r (1,0)^T] = [r (1,0)^T] = [P_{\tilde \alpha} r (1,0)^T] = [A_{\tilde \alpha, r, \theta} v].
    $$

\end{itemize}

\end{proof}

\begin{theorem}\label{tmon}
Let $E \in \mathbb{R}$. If $E \in \sigma_p (H_{\alpha, r, \theta})$, then:
\begin{itemize}
    \item[a)] $E \in \sigma_p(H_{\alpha, r, \tilde \theta})$ if and only if $ \tilde \theta = \theta + k \pi$, $k \in \mathbb{Z}$.
    \item[b)] One of the following holds:
        \begin{itemize}
            \item[$i)$] $E \not\in \sigma_p (H_{\alpha, \tilde r, \theta})$ for every $\tilde r \not= r$.
            \item[$ii)$] $E \in \sigma_p (H_{\alpha, \tilde r, \theta})$ for every $\tilde r > 0$.
        \end{itemize}
    \item[c)] One of the following holds:
    \begin{itemize}
        \item[$i)$] $E \not\in \sigma_p (H_{\tilde \alpha, r, \theta})$ for every $\tilde \alpha \not= \alpha$.
        \item[$ii)$] $E \in \sigma_p (H_{\tilde \alpha, r, \theta})$ for every $\tilde \alpha \in \mathbb R$.
    \end{itemize}
\end{itemize}
\end{theorem}

\begin{proof}
Since $E \in \sigma_p (H_{\alpha, r, \theta})$, there exists $u \in L^2(a,b)$, $u \not= 0$, such that $u \in D(H_{\alpha,r,\theta})$ and $H_{\alpha,r,\theta} u = E u$.
\begin{itemize}

\item[a)] Suppose now that $E \in \sigma_p (H_{\alpha, r, \tilde \theta})$ for some $\tilde \theta \not= \theta$. Then there exists $v \in L^2(a,b)$, $v \not= 0$, such that $v \in D(H_{\alpha, r, \tilde \theta})$ and $H_{\alpha, r, \tilde \theta} v = E v$. We will now consider the matrices $M(p-,a;E)$ and $M(b,p+;E)$, which do not depend on $\alpha, r$ and $\theta$.\\
    Since $M := M(p-,a;E) \in \mathrm{SL}(2,\mathbb{R})$ and $[(u(a),u'(a))^T] = [(v(a),v'(a))^T]$, we have
$$
[(u(p-),u'(p-))^T] = \tilde M([(u(a),u'(a))^T]) = \tilde M([(v(a),v'(a))^T]) = [(v(p-),v'(p-))^T].
$$
Thus there exists $\lambda \in \mathbb{R} \setminus \{ 0 \}$ such that
$$
\begin{pmatrix} u(p-) \\ u'(p-) \end{pmatrix} = \lambda \begin{pmatrix} v(p-) \\ v'(p-) \end{pmatrix}.
$$
Analogously, since $M^{-1}(b,p+;E) \in \mathrm{SL}(2,\mathbb{R})$, there exists $\mu \in \mathbb{R} \setminus \{ 0 \}$ such that
$$
\begin{pmatrix} u(p+) \\ u'(p+) \end{pmatrix} = \mu \begin{pmatrix} v(p+)\\v'(p+) \end{pmatrix}.
$$
Since
$$
\begin{pmatrix} u(p+) \\ u'(p+) \end{pmatrix} = A_{\alpha,r,\theta} \begin{pmatrix} u(p-)\\u'(p-) \end{pmatrix} \qquad \mbox{and} \qquad \begin{pmatrix} v(p+)\\v'(p+) \end{pmatrix} = A_{\alpha,r,\tilde\theta} \begin{pmatrix} v(p-)\\v'(p-) \end{pmatrix},
$$
we have
$$
[A_{\alpha,r,\theta} (u(p-),u'(p-))^T] = [A_{\alpha,r,\tilde\theta} (v(p-),v'(p-))^T] = [A_{\alpha,r,\tilde\theta} (u(p-),u'(p-))^T].
$$
By Lemma~\ref{lem1} this happens if and only if  $\tilde\theta = \theta + k \pi$, $k \in \mathbb{Z}$.

\item[b)] Let us assume that $i)$ is false. Then for some $r_0 \not= r$, there is $E \in \sigma_p (H_{\alpha, r_0, \theta})$. Therefore there exists a non-zero $v \in L^2(a,b)$ such that $v \in D(H_{\alpha, r_0, \theta})$ and $H_{\alpha, r_0, \theta} v = E v$. As in case $a)$ above we conclude
$$
[A_{\alpha, r, \theta} (u(p-),u'(p-))^T] = [A_{\alpha, r_0, \theta} (v(p-),v'(p-))^T] = [A_{\alpha, r_0, \theta} (u(p-),u'(p-))^T].
$$
By Lemma~\ref{lem2} this happens if and only if  $[(u(p-),u'(p-))^T] = [(\sin \theta, \cos \theta)^T]$ or $[(u(p-),u'(p-))^T] = [(\cos \theta, -\sin \theta)^T]$.

Let us assume that $[(u(p-),u'(p-))^T] = [(\sin \theta, \cos \theta)^T]$. If $[(u(p-),u'(p-))^T] = [(\cos \theta, -\sin \theta)^T]$, the argument proceeds analogously. There exists $c \in \mathbb{R} \setminus \{ 0 \}$ such that $(u(p-),u'(p-))^T = c (\sin \theta, \cos \theta)^T$. We normalize and take $c=1$. Let us verify that for each $\tilde r > 0$, $E \in \sigma(H_{\alpha,\tilde r,\theta})$ with eigenvector
\[
w(x) := \left\{ \begin{array}{cc} \frac{\tilde r}{r} u(x) & \mbox{if } a \leq x < p  \\\\ u(x) & \mbox{if } p < x \leq b \end{array} \right.
\]

First notice that $w$ satisfies the conditions at $a$ and $b$ of the functions in $D(H_{\alpha,r,\theta})$ since $u$ satisfies these conditions too. Now
$$
A_{\alpha, \tilde r, \theta} \begin{pmatrix} w(p-) \\ w'(p-) \end{pmatrix} = A_{\alpha, \tilde r, \theta} \left( \frac{\tilde r}{r} \begin{pmatrix} u(p-) \\ u'(p-) \end{pmatrix} \right) = \frac{\tilde r}{r} A_{\alpha, \tilde r, \theta} \begin{pmatrix} \sin \theta \\ \cos \theta \end{pmatrix} = \frac{\tilde r}{r} \frac{1}{\tilde r} \begin{pmatrix} \alpha\\ 1\end{pmatrix} = \frac{1}{r} \begin{pmatrix} \alpha\\ 1\end{pmatrix} = $$
$$
= A_{\alpha,r,\theta} \begin{pmatrix} \sin \theta \\ \cos \theta \end{pmatrix} = \begin{pmatrix} u(p+) \\ u'(p+) \end{pmatrix} = \begin{pmatrix} w(p+)\\w'(p+) \end{pmatrix}.
$$
The first and second equalities hold by definition of $w$ and $u$, the next three equalities are straightforward calculations. The equality before the last one follows because $u \in D(H_{\alpha,r,\theta})$ and the last one follows because $w = u$ to the right of $p$. Therefore $w \in D(H_{\alpha,\tilde r,\theta})$, $\tau w = Ew$ in $[a,b] \backslash \{p\}$, and $E$ is an eigenvalue for $H_{\alpha,\tilde r,\theta}$, $\tilde r>0$.

\item[c)] Let us assume that $i)$ is false. Then for some $\alpha_0 \not= \alpha$, there is $E \in \sigma_p (H_{\alpha_0,r,\theta})$. Therefore there exists $v \in L^2(a,b)$, $v \not= 0$, such that $v \in D(H_{\alpha_0,r,\theta})$ and $H_{\alpha_0,r,\theta} v = E v$. As in case a) above we conclude
$$
[A_{\alpha,r,\theta} (u(p-),u'(p-))^T] = [A_{\alpha_0, r,\theta} (v(p-),v'(p-))^T] = [A_{\alpha_0,r,\theta}(u(p-),u'(p-))^T].
$$
By Lemma~\ref{lem3} this happens if and only if $[(u(p-),u'(p-))^T] = [(\cos\theta,-\sin\theta)^T]$. There exists $c \in \mathbb{R} \setminus \{ 0 \}$ such that $(u(p-),u'(p-))^T = c(\sin\theta,\cos\theta)^T$. We normalize and take $c=1$. For all $\tilde \alpha \in \mathbb{R}$,
$$
A_{\tilde\alpha,r,\theta} \begin{pmatrix} u(p-)\\u'(p-) \end{pmatrix} = A_{\tilde\alpha,r,\theta} \begin{pmatrix} \cos\theta\\-\sin\theta \end{pmatrix} = r \begin{pmatrix} 1&\tilde\alpha\\0&1 \end{pmatrix} \begin{pmatrix} 1\\0 \end{pmatrix} = r \begin{pmatrix} 1\\0 \end{pmatrix}.
$$
Therefore, for every $\tilde \alpha \in \mathbb{R}$, $u \in D(H_{\tilde\alpha,r,\theta})$ and $H_{\tilde\alpha,r,\theta} u = E u$.
\end{itemize}
\end{proof}

%%%%%%%%%%%%%%%%%%%%%%%%%%%%%%%%%%%%%%%%%%%%%%%%%%%%%%%%

\section{The Case of Countably Many General Point Interactions}\label{s.5}

Let $-\infty \leq a < b \leq \infty$ and let $V \in L_\mathrm{loc}^1(a,b)$ be a real-valued function. Fix a set of points  $M = \{ x_n \}_{n \in I} \subset(a,b)$, where $I \subseteq \mathbb{Z}$. We assume that the discrete set $M$ accumulates at most at $a$ or $b$. Let $ \Lambda := \{ \alpha_n \} \subset \mathbb R$, $R := \{ r_n \} \subset (0,\infty)$ and $\Theta := \{ \theta_n \} \subset \mathbb{R}$.

\begin{definition}\label{matrix}
Let $A_{\alpha_n,r_n,\theta_n} \in \mathrm{SL}(2,\mathbb{R})$ with Iwasawa decomposition $A_{\alpha_n, r_n, \theta_n} = P_{\alpha_n} H_{r_n} E_{\theta_n}$, where $P_{\alpha_n} \in \mathcal{P}$, $H_{r_n} \in\mathcal{H}$ and $E_{\theta_n} \in \mathcal{E}$ for every $n \in I$.
\end{definition}

We consider the formal differential expression
$$
\tau := -\frac{d^2}{dx^2} + V.
$$
The maximal operator $T_{\Lambda, R, \Theta}$ is defined by
$$
T_{\Lambda, R, \Theta} f = \tau f
$$
$$
D(T_{\Lambda, R, \Theta}) = \Big\{ f \in L^2(a,b) : \,f,\,f'\mbox{ abs. cont in } (a,b) \backslash M, -f'' + V f \in L^2(a,b),
$$
\[
\left(\begin{array}{c}
     f(x_n+) \\
     f'(x_n+)
\end{array}\right)=A_{\alpha_n,r_n,\theta_n}\left(\begin{array}{c}
f(x_n-)\\
f'(x_n-)\end{array}\right)\forall n\in I \Big\}
\]
	
%we introduce the following definitions.

\begin{definition}
Given $g \in L_\mathrm{loc}^1 (a,b)$ and $z \in \mathbb{C}$, we call $f$ a solution of $(\tau_{\Lambda,R,\Theta} - z) f = g$ if $f$ and $f'$ are absolutely continuous in $(a,b) \backslash M$ with $-f'' + V f - z f = g$ and
\[
\left(\begin{array}{c}
     f(x_n+) \\
     f'(x_n+)
\end{array}\right)=A_{\alpha_n,r_n,\theta_n}\left(\begin{array}{c}
f(x_n-)\\
f'(x_n-)\end{array}\right)\forall n\in I.
\]
\end{definition}

\begin{definition}
We define the Wronskian of two solutions $u_1$ and $u_2$ of $(\tau_{\Lambda,R,\Theta} -z)f=0 $ by
$$
W_x(u_1,u_2) = u_1(x+) u'_2(x+) - u'_1(x+) u_2(x+).
$$
\end{definition}

\begin{definition}
For $f, g \in D(T_{\Lambda, R, \Theta})$, we define the Lagrange bracket by
$$
[f,g]_x = \overline{f(x+)} g'(x+) - \overline{f'(x+)} g(x+).
$$
\end{definition}

The limits $[f,g]_a = \lim_{x \rightarrow a+} [f,g]_x$ and $[f,g]_b = \lim_{x \rightarrow b-} [f,g]_x$ exist; see \cite[Theorem 2.2]{BSW}.\\

A solution of $(\tau_{\Lambda,R,\Theta}-z)f=0$ is said to lie right (resp., left) in $L^2(a,b)$ if $f$ is square integrable in a neighborhood of $b$ (resp., $a$).
\begin{definition}
\hfill
\begin{itemize}

\item[$i)$] $\tau_{\Lambda, R, \Theta}$ is in the \textbf{limit circle case} (lcc) at $b$ if for every $z \in \mathbb{C}$, all solutions of $(\tau_{\lambda, R, \Theta} - z) f = 0$ lie right in $L^2(a,b)$.

\item[$ii)$] $\tau_{\Lambda,R,\theta}$ is in the \textbf{limit point case} (lpc) at $b$ if for every $z \in \mathbb{C}$, there is at least one solution of $(\tau_{\Lambda, R, \Theta} - z) f = 0$ not lying right in $L^2(a,b)$.
\end{itemize}
The same definition applies to the endpoint $a$.
\end{definition}

According to the \textit{Weyl alternative}, see \cite[Theorem 4.4]{BSW}, we have always either $i)$ or $ii)$.

\medskip

Consider the selfadjoint restriction $H_{\Lambda, R, \Theta}$ of $T_{\Lambda, R, \Theta}$ on $L^2(a,b)$, see \cite[pp. 216]{Pk}, \cite[Section 15]{ESSZ} and \cite[Theorem 2.2]{GeKi}, \cite[Section 3]{SEBA}, \cite[Theorem 1]{ADK},
$$
H_{\Lambda, R, \Theta} f = \tau f
$$
\[
\begin{array}{ccc}
D(H_{\Lambda, R, \Theta}) & = & \left\{ f \in D(T_{\Lambda, R, \Theta}) :
\begin{array}{c}
[v,f]_a = 0 \mbox{     if $\tau_{\Lambda, R, \Theta}$ lcc at $a$ }\\
{[w,f]_b=0 } \mbox{     if $\tau_{\Lambda, R, \Theta}$ lcc at $b$}
\end{array}
\right\},
\end{array}
\]
where $v$ and $w$ are non-trivial solutions of $(\tau_{\Lambda,R,\Theta} - \lambda) v = 0$ near $a$ and near $b$, respectively, $\lambda \in \mathbb{R}$.

\begin{remark}
In \cite[Theorem 5.2]{BSW}, the selfadjoint restrictions are characterized using unitary matrices. The case we are treating corresponds to the particular case of the connecting real selfadjoint boundary conditions.
\end{remark}

\begin{remark}\label{fixed}
Whenever we fix a parameter, we do not write it. For example if we fix $R$ and $\Theta$ we shall just write $H_\Lambda$ and analogously for the other cases.
\end{remark}

\begin{definition}\label{regu}
We say that $\tau_{\Lambda, R, \Theta}$ is regular at $a$ if $a$ is finite, $V \in L^1_\mathrm{loc}[a,b)$ and $a$ is not an accumulation point of $M$. The same definition applies to the endpoint $b$.
\end{definition}

If $\tau_{\Lambda, R, \Theta}$ is regular at $a$, then $\tau_{\Lambda, R, \Theta}$ is lcc at $a$ and the condition $[v,f]_a = 0$ can be replaced by
$$
f(a) \cos \psi + f'(a) \sin \psi = 0
$$
for $\psi \in [0,\pi)$. The same holds for $b$.\\

In the rest of this section we are going to fix the values of $\alpha_n$, $r_n$ and $\theta_n$ for $n \not=n_0$  and vary the parameters just at the point $n_0 \in I$. Set $\alpha = \alpha_{n_0}$, $r = r_{n_0}$ and $\theta = \theta_{n_0}$. The maximal operator will be denoted by $T_{\alpha, r, \theta}$ and its selfadjoint restriction by $H_{\alpha, r, \theta}$.\\

For $\delta$, $\gamma \in [0,\pi)$ and $[c,d] \subset [a,b]$ such that $[c,d] \cap M = \{ x_{n_0} \}$, define the operator
$$
H_{\alpha, r, \theta}^{\delta, \gamma} := H_{\alpha, r, \theta}.
$$
where $H_{\alpha, r, \theta}$ is as in formula \eqref{halfa2} from the previous section with $p = x_{n_0}$ and $I = [c,d]$.

Let $E \in \mathbb{R}$ be fixed and define
$$
P(E) := \{ (\alpha, r, \theta) \in \mathbb{R} \times (0, \infty) \times \mathbb{R} : E \in \sigma_p (H_{\alpha, r, \theta, M}) \}.
$$

\begin{lemma}\label{lmon}
There exist $\delta_0$, $\gamma_0 \in [0,\pi)$ such that if $(\alpha, r, \theta) \in P(E)$, then $E \in \sigma_p (H_{\alpha, r, \theta}^{\delta_0, \gamma_0})$.
\end{lemma}

\begin{proof}
This follows as in Lemma 3.1 of \cite{RRAL}
\end{proof}

\begin{theorem}\label{teoad}
We have the following cases:
\begin{itemize}

\item [$a$)] If $\alpha = \alpha_0$ and $r = r_0$ are fixed, then $\{(\alpha_0, r_0, \theta) \in P(E)\}$ is empty or is countable.

\item[$b)$] If $\alpha = \alpha_0$ and $\theta = \theta_0$ are fixed, then $\{(\alpha_0, r, \theta_0) \in P(E)\}$ has at most one element or $\{(\alpha_0, r, \theta_0) \in P(E) \} = \{ \alpha_0 \} \times (0,\infty) \times \{ \theta_0 \}$.

\item[$c)$] If $r = r_0$ and $\theta = \theta_0$ are fixed, then $\{ (\alpha, r_0, \theta_0) \in P(E) \}$ has at most one element or $\{ (\alpha, r_0, \theta_0) \in P(E) \} = \mathbb{R} \times \{ r_0 \} \times \{ \theta_0 \}$.
\end{itemize}
\end{theorem}

\begin{proof}
\hfill
   \begin{itemize}

       \item [$a)$] Suppose that for some $\theta$, $(\alpha_0, r_0, \theta) \in P(E)$. Then by Lemma~\ref{lmon}, $E \in \sigma_p (H_{\alpha_0, r_0, \theta}^{\delta_0, \gamma_0})$. By Theorem~\ref{tmon} $a)$, this implies $(\alpha_0, r_0, \tilde \theta) \in P(E)$ if and only if $\tilde \theta = \theta + k \pi$, $k \in \mathbb{Z}$. Therefore the set $\{ (\alpha_0, r_0, \theta) \in P(E) \}$ is countable.
       \item[$b)$] Suppose that for some $r$, $(\alpha_0, r, \theta_0) \in P(E)$. Then by Lemma~\ref{lmon}, $E \in \sigma_p (H_{\alpha_0, r, \theta_0}^{\delta_0, \gamma_0})$. By Theorem~\ref{tmon} $b)$, one has $(\alpha_0, \tilde r, \theta _0) \not\in P(E)$, $\forall \tilde r\not= r$ or $(\alpha_0, \tilde r, \theta _0) \in P(E)$, $\forall \tilde r > 0$. Therefore the assertion follows.
       \item[c)] Suppose that for some $\alpha$, $(\alpha, r_0, \theta_0) \in P(E)$. Then by Lemma~\ref{lmon}, $E \in \sigma_p (H_{\alpha, r_0, \theta_0}^{\delta_0, \gamma_0})$. By Theorem~\ref{tmon} $c)$, one has $(\tilde \alpha, r_0, \theta _0) \not\in P(E)$, $\forall \tilde \alpha \not= \alpha$ or $(\tilde \alpha, r_0, \theta _0) \in P(E)$, $\forall \tilde \alpha \in \mathbb{R}$. Therefore the assertion follows.

   \end{itemize}
\end{proof}	

\begin{remark}\label{remad}
Observe that in $b)$ of Theorem~\ref{teoad}, if the eigenvector associated to $E$ is such that $u(x_{n_0}-) = \cos \theta_{n_0}$ and  $u'(x_{n_0}-) = -\sin \theta_{n_0}$ or $u(x_{n_0}-) = \sin \theta_{n_0}$ and $u'(x_{n_0}-) = \cos \theta_{n_0}$, then $\{ (\alpha_0, r, \theta_0) \in P(E) \} = \{ \alpha_0 \} \times (0, \infty) \times \{ \theta_0 \}$, otherwise $\{ (\alpha_0, r, \theta_0) \in P(E) \}$ has at most one element. In case $c)$ of the same theorem, if the eigenvector associated to $E$ is such that $u(x_{n_0}-) = \cos \theta_{n_0}$ and $u'(x_{n_0}-) = -\sin \theta_{n_0}$, then $\{ (\alpha, r_0, \theta_0) \in P(E) \} = \mathbb{R} \times \{ r_0 \} \times \{ \theta_0 \}$, otherwise $\{ (\alpha, r_0, \theta_0) \in P(E) \}$ has at most one element.
\end{remark}

%%%%%%%%%%%%%%%%%%%%%%%%%%%%

\section{Sturm-Liouville Operators with Random Point Interactions}\label{s.6}

In this section we use the previously obtained results to study the random case.
First the probability space $\Omega$ where the sequences of coupling constants live is constructed and then our random operators are defined.\\

The space of real valued sequences $\{\omega_n\}_{n\in I}$, where $I\subseteq \mathbb{Z}$, will be denoted by $\mathbb{R}^I$. We introduce a measure in $\mathbb{R}^I$ in the following way. Let $\{p_n\}_{n\in I}$ be a sequence of probability measures in $\mathbb{R}$ and consider the product measure  $\mathbb{P}=\times_{n\in I}p_n$ defined on the product $\sigma$-algebra $\mathcal{F}$ of $\mathbb{R}^I$ generated by the cylinder sets, that is, by the sets of the form $\{\omega:\omega(i_1)\in A_1,\dots,\omega(i_n)\in A_n\}$ for $i_1,\dots,i_n\in I$, where $A_1,\dots, A_n$ are Borel sets in $\mathbb{R}$. In this way a measure space $\Omega=(\mathbb{R}^I,\mathcal{F},\mathbb{P})$ is constructed. %%We consider then the completion of this space (subsets of sets of measure zero are measurable) $\tilde\Omega$ which will be denoted by $\Omega$.
See chapter 1, section 1 in \cite{PF}.  In some cases we may require for the measure space $\Omega$ to be complete, i.e. subsets of sets of measure zero are measurable. Every measurable space can be completed, see Theorem 1.36 \cite{WR}.
\\

If we fix $R$ and $\Theta$, and let $\Lambda\in\mathbb{R}^I$, we denote the operator $H_{\Lambda,R,\Theta}$ as $H_\Lambda$ and analogously $H_R$ and $H_\Theta$ when the parameters $\Lambda$ and $\Theta$ or $\Lambda$ and $R$ are fixed respectively, see Remark \ref{fixed}. Assume moreover the limit point occurs at $a$ or that $\tau_{\Lambda,R,\Theta}$ is regular at $a$ and the same possibilities for $b$ (see Definition \ref{regu}).\\

Let $\Omega_1=(\mathbb{R}^I,\mathcal{F}_1,\mathbb{P}_1)$, $\Omega_2=((0,\infty)^I,\mathcal{F}_2,\mathbb{P}_2)$ and $\Omega_3=(\mathbb{R}^I,\mathcal{F}_3,\mathbb{P}_3)$ be probability spaces constructed as described above.

\begin{definition}\label{pes}
    For any $E\in\mathbb{R}$, we define
    $$P_{R,\Theta}(E):=\{\Lambda\in\Omega_1: E\in\sigma_p(H_\Lambda )\}$$
    $$P_{\Lambda,\Theta}(E):=\{R \in\Omega_2: E\in\sigma_p(H_R )\}$$
    $$P_{\Lambda,R}(E):=\{\Theta\in\Omega_3: E\in\sigma_p(H_\Theta )\}$$
\end{definition}

We shall prove the following theorem.
\begin{theorem}\label{isornot}

 Assume $\Omega_1$ is complete and $\mathbb{P}_1= \times_{n\in I}p_n$ is such that $p_n$ are continuous measures for all $n\in I$. Let $E\in \mathbb{R}$ fixed and $B$ any measurable subset of $P_{R,\Theta}(E)$. Then one of the following options hold:
\begin{itemize}
	\item[$i)$] $\mathbb{P}_1(B)=0$
	\item[$ii)$] $P_{R,\Theta}(E)=\Omega_1$
\end{itemize}
\end{theorem}

\begin{remark}
We will show that in some cases there is always a set of point interactions $M$ where option $ii)$ happens. See Theorem \ref{teotal} below.
\end{remark}

\begin{remark}
An analogous result holds for $P_{\Lambda,\Theta}(E)$.
\end{remark}

Before proving Theorem \ref{isornot}  we shall prove the following lemma, where Definition \ref{eqqclass} is used.

\begin{lemma}\label{qmes}
For any measurable $B \subseteq P_{R, \Theta}$ and any $n \in I$, set
$$
Q_{n,E} := \{ \Lambda \in B : \exists u_\Lambda \in D(H_\Lambda) \setminus \{ 0 \},\, H_\Lambda u_\Lambda = E u_\Lambda \mbox{ and } [(u_\Lambda(x_n-) ,u'_\Lambda(x_n-))^T]\not= [(\cos \theta_n, -\sin \theta_n)^T] \}.
$$
Then $Q_{n,E}$ is measurable and $\mathbb{P}_1(Q_{n,E})=0$.
\end{lemma}

\begin{proof}
Let
\[
\begin{array}{ccc}
\chi_B(\Lambda) & = & \left\{ \begin{array}{cc} 1 & \mbox{ if } \Lambda \in B, \\ 0 & \mbox{ if } \Lambda \not\in B. \end{array} \right.
\end{array}
\]

If $\Lambda \in Q_{n,E}$, then from the definition of $Q_{n,E}$ it follows that $\chi_B(\Lambda)=1$. \\

Let $f : \mathbb{R}^{I \backslash \{n\}} \rightarrow [0,\infty)$ be defined by
$$
f(\tilde\Lambda) := \int_\mathbb{R} \chi_B(\Lambda) \, dp_{n}(\Lambda(n)),
$$
where $\tilde \Lambda = \sum \limits_{k \in I \backslash \{n\}} \Lambda(k) e(k)$. Here $e(k) = (e_m)_{m \in I}$ are the canonical vectors with entries $e_m = 0$ if $k \not= m$ and $e_k = 1$. The measurability of $f$ follows from Fubini's Theorem. (See \cite[Theorem 7.8]{WR}.)

If $\Lambda = \sum \limits_{k \in I} \Lambda(k) e(k) \in Q_{n,E}$, then $f(\tilde \Lambda) = 0$, where $\tilde \Lambda = \sum \limits_{k \in I \backslash \{n\}} \Lambda(k) e(k)$. This follows from Remark~\ref{remad} since $p_n$ is continuous. \\

Hence $Q_{n,E} \subseteq [f^{-1}(\{0\}) \times \mathbb{R}] \cap B$.\\

Now, using Fubini,
$$
\int_{f^{-1}(\{0\}) \times \mathbb{R}} \chi_B(\Lambda) \, d\mathbb{P}_1 = \int_{f^{-1}(\{0\})} \, d\mathbb{P}_1(\tilde\Lambda) \int_\mathbb{R} \chi_B(\Lambda) \, dp_n(\Lambda(n)) = \int_{f^{-1}(\{0\})} f(\tilde \Lambda) \, d\mathbb{P}_1(\tilde \Lambda) = 0.
$$

Then,
$$
\int_{[f^{-1}(\{0\}) \times \mathbb{R}] \cap B} \chi_B(\Lambda) \, d\mathbb{P}_1 = 0,
$$
and since $\chi_B(\Lambda)=1$ in $B$, we get $\mathbb{P}_1([f^{-1}(\{0\}) \times \mathbb{R}] \cap B) = 0$.\\

Since the measure $d\mathbb{P}_1$ is complete, any subset of a measurable set of measure zero is measurable with measure zero. Therefore $Q_{n,E}$ is measurable.
\end{proof}

\begin{proof}[Proof of Theorem \ref{isornot}]
It will be enough to prove that if $ii)$ doesn't hold, then $i)$ must hold.\\

Assume that there exists $\Lambda_0 \in \Omega_1$ such that $E$ is not an eigenvalue of $H_{\Lambda_0}$. \\

If $E$ is not an eigenvalue of $H_\Lambda$ for every $\Lambda \in \Omega_1$, then $\mathbb P_1(B) = 0$ and the result follows. \\

Suppose now $\Lambda \in B$, then $E\in\sigma_p(H_\Lambda)$, i.e. there exist $u_\Lambda\in D(H_\Lambda) \setminus \{ 0 \}$ such that $H_\Lambda u_\Lambda = E u_\Lambda$. Then $\Lambda \in Q_{n,E}$ for some $n\in I$. This follows because if $[(u_\Lambda(x_n-) ,u'_\Lambda(x_n-))^T]= [(\cos \theta_n, -\sin \theta_n)^T]$ for every $n\in I$, then there exist $c_n\in\mathbb{R}$ such that  $(u(x_n-),u'(x_n-))^T=c_n(\cos\theta,-\sin\theta)$, hence
$$
A_{\Lambda(n), r_n, \theta_n} \begin{pmatrix} u(x_n-) \\ u'(x_n-) \end{pmatrix} = A_{\Lambda(n), r_n, \theta_n}c_n \begin{pmatrix} \cos \theta_n \\ -\sin \theta_n \end{pmatrix} = c_nr_n \begin{pmatrix} 1 & \Lambda(n) \\ 0 & 1 \end{pmatrix} \begin{pmatrix} 1 \\ 0 \end{pmatrix} =c_n r_n \begin{pmatrix} 1 \\ 0 \end{pmatrix}
$$
Since the right hand side does not depend on $\Lambda$, from the definition of $H_\Lambda$, $E$ must be an eigenvalue of $H_\Lambda$ for all $\Lambda \in \Omega_1$,  in particular $E$ is an eigenvalue of $H_{\Lambda_0}$, cf. proof Theorem~\ref{tmon} $c)$,  which is not possible by our initial assumption. Therefore
$$
B \subset \bigcup_{n \in I} Q_{n,E},
$$
where $Q_{n,E}$ was defined in Lemma~\ref{qmes}. Using that lemma we obtain $\mathbb{P}_1(\bigcup \limits_{n \in I} Q_n) = 0$. Therefore the result follows.
\end{proof}

\begin{theorem}\label{isoris}
Assume $\mathbb{P}_3 = \times_{n \in I}q_n$ is such that $q_{n_0}$ is a continuous measure for some $n_0 \in I$. Let $E \in \mathbb{R}$ be fixed and let $B$ be any measurable subset of $P_{\Lambda, R}(E)$. Then $\mathbb P_3 (B) = 0$.
\end{theorem}

\begin{proof}
Let
\[
\begin{array}{ccc} \chi_B(\Theta) & = & \left\{ \begin{array}{cc} 1 & \mbox{ if } \Theta \in B, \\ 0 & \mbox{ if } \Theta \not\in B, \end{array} \right. \end{array}
\]
and define $f : \mathbb{R}^{I \backslash \{n_0\}} \rightarrow [0,\infty)$ as
$$
f(\tilde \Theta) := \int_\mathbb{R} \chi_B(\Theta) \, dq_{n_0}(\Theta(n_0)),
$$
where $\tilde \Theta = \sum \limits_{k \in I \backslash \{n_0\}} \Theta(k) e(k)$. Here $e(k) = (e_m)_{m \in I}$ are the canonical vectors with entries $e_m = 0$ if $k \not= m$ and $e_k=1$. The measurability of $f$ follows from Fubini's Theorem. (See \cite[Theorem 7.8]{WR}.)

If $\Theta = \sum \limits_{k \in I} \Theta(k) e(k) \in B$, then $f(\tilde \Theta) = 0$, where $\tilde \Theta = \sum \limits_{k \in I \backslash \{n\}} \Theta(k) e(k)$. This follows from Theorem~\ref{teoad} since $q_{n_0}$ is continuous.\\

Hence $B\subseteq [f^{-1}(\{0\}) \times \mathbb{R}]$.\\

Now, using  Fubini,
$$
\int_{f^{-1}(\{0\}) \times \mathbb{R}} \chi_B(\Theta) \, d\mathbb{P}_3 = \int_{f^{-1}(\{0\})} \, d\mathbb{P}_3(\tilde \Theta) \int_\mathbb{R} \chi_B(\Theta) \, dq_{n_0}(\Theta(n_0)) = \int_{f^{-1}(\{0\})} f(\tilde \Theta) \, d\mathbb{P}_3(\tilde \Theta) = 0.
$$

Then, $\mathbb{P}_3([f^{-1}(\{0\}) \times \mathbb{R}]) = 0$. Therefore $\mathbb{P}_3(B) = 0$.
\end{proof}

\begin{definition}
For any $E \in \mathbb{R}$, we define
$$
P(E) := \{ (\Lambda, R, \Theta) \in \Omega_1 \times \Omega_2 \times \Omega_3 : E \in \sigma_p (H_{\Lambda, R, \Theta}) \}.
$$
\end{definition}

\begin{theorem}
Assume $\mathbb{P}_3 = \times_{n \in I}q_n$ is such that $q_{n_0}$ is a continuous measure for some $n_0 \in I$. Let $E \in \mathbb{R}$ be fixed and suppose that $P(E)$ is measurable. Let $\mathbb{P} = \mathbb{P}_1 \times \mathbb{P}_2 \times \mathbb{P}_3$. Then,
$$
\mathbb{P}(P(E)) = 0.
$$
\end{theorem}

\begin{proof}
Let
\[
\begin{array}{ccc} \chi_{P(E)} (\Lambda, R, \Theta) & = & \left\{ \begin{array}{cc} 1 & \mbox{ if }(\Lambda, R, \Theta) \in P(E), \\ 0 & \mbox{ if }(\Lambda, R, \Theta) \not\in P(E). \end{array} \right. \end{array}
\]
Then,
$$
\mathbb{P}(P(E)) = \int_{\Omega_1 \times \Omega_2 \times \Omega_3} \chi_{P(E)}(\Lambda, R, \Theta) \, d\mathbb{P}.
$$
Using Fubini we have
$$
\int_{\Omega_1 \times \Omega_2 \times \Omega_3} \chi_{P(E)}(\Lambda, R, \Theta) \, d\mathbb{P} = \int_{\Omega_1 \times \Omega_2} \, d\mathbb{P}_1 \times d\mathbb{P}_2 \int_{\Omega_3} \chi_{P_{\Lambda, R}(E)}(\Theta) \, d\mathbb{P}_3(\Theta),
$$
where $P_{\Lambda,R}(E)$ is as in Definition~\ref{pes}.

Note that
$$
\int_{\Omega_3} \chi_{P_{\Lambda, R}(E)}(\Theta) \, d\mathbb{P}_3(\Theta) = \mathbb{P}_3(P_{\Lambda,R}(E)),
$$
and that Theorem~\ref{isoris} gives $\mathbb{P}_3(P_{R,\Theta}(E)) = 0$. Thus, the theorem follows.
\end{proof}

\begin{theorem}
Assume $\Omega_1$ is complete and $\mathbb{P}_1 = \times_{n \in I} p_n$ is such that $p_n$ are continuous measures for all $n \in I$. Let $E \in \mathbb{R}$ be fixed and suppose that $P(E)$ is measurable. Let $\mathbb{P} = \mathbb{P}_1 \times \mathbb{P}_2 \times \mathbb{P}_3$. Then one of the following options holds:
\begin{itemize}
    \item[$i)$] $\mathbb{P}(P(E)) = 0$,
    \item[$ii)$] $\mathbb{P}(P(E)) = 1$.
\end{itemize}
\end{theorem}

\begin{proof}
Let
\[
\begin{array}{ccc} \chi_{P(E)}(\Lambda, R, \Theta) & = & \left\{ \begin{array}{cc} 1 & \mbox{ if }(\Lambda, R, \Theta) \in P(E), \\ 0 & \mbox{ if }(\Lambda, R, \Theta) \not\in P(E). \end{array}\right.
\end{array}
\]
Then,
$$
\mathbb{P}(P(E)) = \int_{\Omega_1 \times \Omega_2 \times \Omega_3} \chi_{P(E)}(\Lambda, R, \Theta) \, d\mathbb{P}.
$$
Using Fubini we have
$$
\int_{\Omega_1 \times \Omega_2 \times \Omega_3} \chi_{P(E)}(\Lambda, R, \Theta) \, d\mathbb{P} = \int_{\Omega_2 \times \Omega_3} \, d\mathbb{P}_2 \times d\mathbb{P}_3 \int_{\Omega_1} \chi_{P_{R,\Theta}(E)}(\Lambda) \, d\mathbb{P}_1(\Lambda),
$$
where $P_{R, \Theta}(E)$ is as in Definition~\ref{pes}.
Since
$$
\int_\Omega \chi_{P_{R, \Theta}(E)}(\Lambda) \, d\mathbb{P}(\Lambda) = \mathbb{P}(P_{R, \Theta}(E)),
$$
using Theorem~\ref{isornot} we conclude that either $\mathbb{P}(P_{R, \Theta}(E)) = 0$ or $\mathbb{P}(P_{R, \Theta}(E)) = 1$. Therefore the theorem follows.
\end{proof}

\begin{remark}
An analogous result holds if we assume that $\Omega_2$ satisfies the hypothesis of the theorem instead of $\Omega_1$.
\end{remark}

\subsection{Oscillation of Solutions}

The next result, Theorem~\ref{teotal}, shows that it is always possible to construct a set of point interactions $M$ such that option $ii)$ in Theorem~\ref{isornot} occurs.\\
Let $H = H_{\Lambda, R, \Theta}$ with $\Lambda = \{0\}_{n \in I}$, $R = \{1\}_{n \in I}$ and $\Theta = \{0\}_{n \in I}$ be the unperturbed operator. This operator does not depend on $M$ and $I$, and it is just the classical selfadjoint operator without interactions.

%\begin{lemma}
%Let $E\in\sigma_p(H)$ with $H u=Eu$ and $M=\{x_n\}_{n\in I}$ a discrete set. Let us denote $\theta_n:=\arctan \left(-\frac{u'(x_n)}{u(x_n)}\right)$. If $\Theta:=\{\theta_n\}_{n\in I}$, then $P_{R,\Theta}(E)=\Omega_1$, for all $R\in\Omega_2$.
%\end{lemma}
%\begin{proof}
%From the definition of $\theta_n$ follows that
%$$\left[\begin{pmatrix}u(x_n)\\u'(x_n)\end{pmatrix}\right]=\left[\begin{pmatrix}\cos\theta_n\\-\sin\theta_n\end{pmatrix}\right]\qquad\qquad\forall n\in I$$
%Take $A_{\alpha_n,r_n,\theta_n}=A_{\alpha_n}A_{r_n}A_{\theta_n}$ as in Definition \ref{matrix}, for all $n\in I$. Then

%$$A_{\alpha_n,r_n,\theta_n}\begin{pmatrix}u(x_n-)\\u'(x_n-) \end{pmatrix}=A_{\alpha_n,r_n,\theta_n}\begin{pmatrix}\cos\theta_n\\-\sin\theta_n \end{pmatrix}=r_n\begin{pmatrix}1&\alpha_n\\0&1\end{pmatrix}\begin{pmatrix}1\\0\end{pmatrix}=r_n\begin{pmatrix}1\\0\end{pmatrix}.$$
 %From the definition of $H_\Lambda$, $E$ must be an eigenvalue of $H_\Lambda$ for all $\Lambda\in\Omega_1$ and therefore $P_{R,\Theta}(E)=\Omega_1$.
%\end{proof}

\begin{definition}[See Section XI.6 in \cite{HP}]
The equation
$$
(\tau-E) u = 0
$$
is said to be oscillatory on an interval $J$ if every solution has infinitely many zeros on $J$.\\
	
If $t = b$ is a (possibly infinite) endpoint of $J$ which does not belong to $J$, then the equation is said to be oscillatory at $t=b$ if every solution has an infinite number of zeros in $J$ accumulating at $b$.
\end{definition}	

Define
$$
\varphi(x) := arg(u'(x) + i u(x)) \qquad \qquad x \in(a,b).
$$
The zeros of the solution $u$ are given by the values of $x$ such that $\varphi(x) = k \pi$ for some integer $k$. $(\tau-E) u = 0$ is oscillatory at $b$ if and only if $\varphi(x) \rightarrow \infty$ as $x \rightarrow b$; see \cite[p.9]{DHU}.

\begin{lemma}\label{angle}
Given two consecutive zeros $t_1, t_2 \in (a,b)$ of a solution $u$ of $(\tau-E) u = 0$ and given a vector $v = (v_1,v_2)^T \in \mathbb{R}^2$, there exists a point $x_0 \in [t_1,t_2)$ such that
$$
\left[ \begin{pmatrix} u(x_0)\\u'(x_{0}) \end{pmatrix} \right] = \left[ \begin{pmatrix} v_1\\v_2 \end{pmatrix} \right].
$$
\end{lemma}

\begin{proof}
Since $t_1$ and $t_2$ are zeros of the solution $u$, there exist $k_1, k_2 \in \mathbb{Z}$ such that $\varphi(t_1) = k_1 \pi$ and $\varphi(t_2) = k_2 \pi$. Since $\varphi$ cannot tend to a multiple of $\pi$ from above, see \cite[Theorem 8.4.3 $ii)$]{ATK}, we have $k_2 = k_1 + 1$. Since $\varphi$ is continuous, there exists $x_0 \in [t_1,t_2)$ such that
$$
\mathrm{arg}(u'(x_0) + i u(x_0)) = \mathrm{arg}(v_2 + i v_1).
$$
Therefore, by Remark~\ref{linearg},
$$
\left[ \begin{pmatrix} u(x_0)\\u'(x_{0}) \end{pmatrix} \right] = \left[ \begin{pmatrix} v_1\\v_2 \end{pmatrix} \right].
$$
\end{proof}

\begin{theorem}\label{teotal}
Let $(\tau-E) u = 0$ be oscillatory and $E \in \sigma_p(H)$. Fix $R = \{r_n\}_{n \in I}$ and $\Theta = \{\theta_n\}_{n \in I}$, where $I$ is finite or $I = \mathbb{N}$. Then there exists $M \subset \mathbb{R}$ discrete such that $P_{R,\Theta}(E) = \Omega_1$.
\end{theorem}

\begin{proof}
Assume $Hu = Eu$.\\
Suppose $I$ is finite, $I = \{n_1,n_2,\dots,n_r\}$, and $\Theta = \{ \theta_{n_1}, \dots, \theta_{n_r} \}$. Let $t_0, t_1, \dots, t_r$ be $r+1$ consecutive zeros of $u$. For $n_i \in I$, let $x_{n_i}$ be such that $x_{n_i} \in [t_{i-1},t_i)$ and
$$
\left[ \begin{pmatrix} u(x_{n_i})\\u'(x_{n_i}) \end{pmatrix} \right] = \left[ \begin{pmatrix} \cos \theta_{n_i}\\-\sin \theta_{n_i} \end{pmatrix} \right].
$$
Due to Lemma~\ref{angle}, such an $x_{n_i}$ exists. Let $M = \{x_{n_i}\}_{i=1}^r$ and take $A_{\alpha_{n_i}, r_{n_i}, \theta_{n_i}} = P_{\alpha_{n_i}} H_{r_{n_i}} E_{\theta_{n_i}}$ as in Definition~\ref{matrix}, for all $n_i \in I$. Then,
$$
A_{\alpha_{n_i}, r_{n_i}, \theta_{n_i}} \begin{pmatrix} u(x_{n_i}-)\\u'(x_{n_i}-) \end{pmatrix} = A_{\alpha_{n_i}, r_{n_i}, \theta_{n_i}} \begin{pmatrix} \cos \theta_{n_i} \\ -\sin \theta_{n_i} \end{pmatrix} = r_{n_i} \begin{pmatrix} 1  &\alpha_{n_i} \\ 0 & 1 \end{pmatrix} \begin{pmatrix} 1 \\ 0 \end{pmatrix} = r_{n_i} \begin{pmatrix} 1 \\ 0 \end{pmatrix}.
$$
From the definition of $H_\Lambda$, $E$ must be an eigenvalue of $H_\Lambda$ for all $\Lambda \in \Omega_1$, and therefore $P_{R, \Theta}(E) = \Omega_1$.\\
Suppose $I = \mathbb{N}$. Let us assume that there are infinitely many zeros of $u$, increasingly enumerated by $t_0, t_1, \dots$. Let $\Theta = \{\theta_n\}_{n \in I}$. Let $x_{1}$ be such that $x_{1} \in [t_{0},t_1)$ and
$$
\left[ \begin{pmatrix} u(x_{1})\\u'(x_{1}) \end{pmatrix} \right] = \left[ \begin{pmatrix} \cos \theta_{1}\\-\sin\theta_{1} \end{pmatrix} \right].
$$

As above let $x_{2}$ be such that $x_{2} \in [t_{1},t_2)$ and
$$
\left[ \begin{pmatrix} u(x_{2})\\u'(x_{2}) \end{pmatrix} \right] = \left[ \begin{pmatrix} \cos \theta_{2}\\-\sin \theta_{2} \end{pmatrix} \right].
$$

%Take $n_{-1}\in I$ be the greatest point in $I$ less than $n_0$.

In this way we get a sequence $M := \{x_{n}\}_{n \in I}$. Take $A_{\alpha_{n}, r_{n}, \theta_{n}} = P_{\alpha_{n}} H_{r_{}} E_{\theta_{n}}$ as in Definition~\ref{matrix}. Then,
$$
A_{\alpha_{n}, r_{n}, \theta_{n}} \begin{pmatrix} u(x_{n}-) \\ u'(x_{n}-) \end{pmatrix} = A_{\alpha_{n}, r_{n}, \theta_{n}} \begin{pmatrix} \cos \theta_{n} \\ -\sin \theta_{n} \end{pmatrix} = r_{n} \begin{pmatrix} 1 & \alpha_{n} \\ 0 & 1 \end{pmatrix} \begin{pmatrix} 1 \\ 0 \end{pmatrix} = r_{n} \begin{pmatrix} 1 \\ 0 \end{pmatrix}.
$$
From the definition of $H_\Lambda$, $E$ must be an eigenvalue of $H_\Lambda$ for all $\Lambda \in \Omega_1$, and therefore $P_{R,\Theta}(E) = \Omega_1$.
\end{proof}

\bibliographystyle{plain}
\bibliography{biblio}

\end{document}